\documentclass[11pt,
reqno]{amsart}
\usepackage{latexsym,amsmath,amssymb
}
 \newtheorem{thm}{Theorem}[section]
 \newtheorem{cor}[thm]{Corollary}
 
 \newtheorem{prop}[thm]{Proposition}
 \theoremstyle{definition}
 \newtheorem{defn}[thm]{Definition}
 \theoremstyle{remark}
 \newtheorem{rem}[thm]{Remark}
 \newtheorem{ex}[thm]{Example}
 \numberwithin{equation}{section}

\newcommand{\1}{{\bf 1}}

\newcommand{\Ad}{{\rm Ad}}

\newcommand{\Aut}{{\rm Aut}}
\newcommand{\Ci}{{\mathcal C}^\infty}
\newcommand{\Diff}{{\rm Diff}\,}
\newcommand{\de}{{\rm d}}
\newcommand{\ee}{{\rm e}}
\newcommand{\ie}{{\rm i}}
\newcommand{\OP}{{\rm OP}}
\newcommand{\Op}{{\rm Op}}
\newcommand{\Sp}{{\rm Sp}}
\renewcommand{\sp}{{{\mathfrak s}{\mathfrak p}}}
\newcommand{\spa}{{\rm span}}
\newcommand{\sym}{{\rm sym}}
\newcommand{\Tr}{{\rm Tr}\,}

\newcommand{\Bc}{{\mathcal B}}
\newcommand{\Cc}{{\mathcal C}}
\newcommand{\Fc}{{\mathcal F}}
\newcommand{\Hc}{{\mathcal H}}
\newcommand{\Lc}{{\mathcal L}}
\newcommand{\Oc}{{\mathcal O}}
\newcommand{\Sc}{{\mathcal S}}
\newcommand{\Tc}{{\mathcal T}}
 
\renewcommand{\gg}{{\mathfrak g}}
\newcommand{\hg}{{\mathfrak h}}
\newcommand{\zg}{{\mathfrak z}}

\newcommand{\CC}{{\mathbb C}}
\newcommand{\RR}{{\mathbb R}}

\begin{document}

%
%
%
%
%
%
%
%
%

\title[Nilpotent Lie groups]
 {Boundedness for Pseudo-Differential\\ Calculus on Nilpotent Lie Groups}

\author[I. Belti\c t\u a]{Ingrid Belti\c t\u a}

\address{%
Institute of Mathematics  
``Simion Stoilow''  
of the Romanian Academy, 
P.O. Box 1-764, 
Bucharest, 
Romania}

\email{Ingrid.Beltita@imar.ro}

\thanks{This research has been partially supported by the Grant
of the Romanian National Authority for Scientific Research, CNCS-UEFISCDI,
project number PN-II-ID-PCE-2011-3-0131.}

\author[D. Belti\c t\u a]{Daniel Belti\c t\u a}

\address{%
Institute of Mathematics  
``Simion Stoilow''  
of the Romanian Academy, 
P.O. Box 1-764, 
Bucharest, 
Romania}

\email{Daniel.Beltita@imar.ro}

\author[M. Pascu]{Mihai Pascu}

\address{%
University ``Petrol-Gaze'' of Ploie\c sti 
and
Institute of Mathematics
``Simion Stoilow''
of the Romanian Academy, 
P.O. Box 1-764, 
Bucharest
Romania}

\email{Mihai.Pascu@imar.ro}



\date{September 29, 2012}

\begin{abstract}
We survey a few results on the boundedness of  operators arising from the Weyl-Pedersen  calculus
associated with irreducible representations of nilpotent Lie groups. \end{abstract}

\maketitle

\section{Introduction}

Convolution operators and global Weyl calculus on nilpotent Lie groups have been extensively studied
in many papers, in connection with various problems in partial differential equation and representation theory
(see \cite{CW73, Gl04, Gl07, Go02, How84, Ma91, Me83, Mi82, Mi86, Ra85}). 
It has been repeatedly remarked that the global Weyl calculus is an extension of the classical Weyl calculus on $\RR^n$; 
however, to see this, some further identifications and results are needed (see \cite{Mi82}).
This phenomenon has roots in the fact that the global Weyl calculus is not injective, and thought it is associated to a given 
irreducible representation, the link with the corresponding coadjoint orbit is not clear. 
These issues were resolved by N.V.~Pedersen in \cite{Pe94}, who constructed a Weyl calculus ---that we call Weyl-Pedersen calculus--- associated to an irreducible 
representation of a nilpotent Lie group, which is a bijection between good function spaces of symbols defined on the corresponding orbit and operators defined on the Hilbert space of the representation. 
In addition, this calculus directly extends the classical Weyl calculus.

The aim of this paper is to survey some boundedness results for the Weyl-Pedersen calculus in the case of flat orbits and to give further applications to some three-step nilpotent Lie groups that have non-flat generic orbits. 
The results are generalizations of the classical Calder\'on-Vaillancourt theorem \cite{CV72} and of the Beals characterization of the pseudo-differential operators \cite{Be77} (see also \cite{Bo97}).

Main definitions are given in Section~\ref{section2}, along with an example illustrating the non-injectivity 
of the global Weyl calculus.
The boundedness results are given in Section~\ref{section3}.

Finally, let us mention here that other extensions of the classical Weyl calculus have been constructed
in terms of representations,
for instance the magnetic calculus on $\RR^n$ \cite{IMP07}, and nilpotent Lie groups \cite{BB11a}, Weyl calculus on nilpotent $p$-adic Lie groups \cite{Ha93}, \cite{Bec04}.
 
A good source for the background information on nilpotent Lie groups and their representations is \cite{CG90}. 

\section{Weyl calculi for representations of nilpotent Lie groups}\label{section2}

\subsection*{Preliminaries on nilpotent Lie groups}
Throughout this paper the nilpotent Lie groups are supposed to be connected and simply connected. 
Therefore there is no loss of generality in assuming that a nilpotent Lie group is a pair $G=(\gg,\cdot)$, 
where $\gg$ is a nilpotent Lie algebra (over $\RR$ unless otherwise mentioned) with the Lie bracket 
$[\cdot,\cdot]$, 
and the group multiplication $\cdot$ is given by the Baker-Campbell-Hausdorff series: 
$$(\forall x,y\in\gg)\quad x\cdot y=x+y+\frac{1}{2}[x,y]+\frac{1}{12}([x,[x,y]]+[y,[y,x]])+\cdots$$ 
If the  Lie algebra $\gg$ is nilpotent of step $n$,  the group multiplication $\gg\times\gg\to\gg$, $(x,y)\mapsto x\cdot y$ 
is a polynomial mapping of degree~$n$.
With this identification the exponential from the Lie algebra to the group is then the identity, while the inverse of $x\in \gg$ is $-x$, and the unit element is $0\in \gg$.

We recall that for every $\xi\in\gg^*$ the corresponding coadjoint orbit is 
$$\Oc_\xi:=\{(\Ad^*_G x)\xi\mid x\in \gg\}\simeq G/G_\xi$$
where  $G_\xi:=\{x\in\gg\mid(\Ad^*_G x)\xi=\xi \}$ is the coadjoint isotropy group (or the stabilizer of $\xi$), and 
$\Ad^*_G\colon G\times\gg^*\to\gg^*$ stands for the coadjoint action.
We use the notation $\gg_\xi$ for the corresponding Lie algebra, called the radical of $\xi$.

We will always denote by $\de x$ a fixed Lebesgue measure on $\gg$ and we recall that this is also a two-sided Haar measure on the group $G$.
We let  $\de\xi$ be a Lebesgue measure on $\gg^*$ with the property that if we define the Fourier transform
for every $a\in L^1(\gg^*)$ by 
$$(\Fc a)(x)=\int\limits_{\gg^*}\ee^{-\ie\langle\xi,x\rangle}a(\xi)\de\xi,$$
then we get a unitary operator $\Fc\colon L^2(\gg^*,\de\xi)\to L^2(\gg,\de x)$. 
We also denote by $\Sc(\gg)$ ( $\Sc(\gg^*))$ the Schwartz space on $\gg$ (respectively $\gg^*$), by $\Sc'(\gg)$ ($\Sc'(\gg^*))$) 
 its topological dual consisting of the tempered distributions, 
and by $\langle\cdot,\cdot\rangle\colon\Sc'(\gg)\times\Sc(\gg)\to\CC$ the corresponding duality pairing. 
The Fourier transform extends to a linear topological isomorphism $\Fc\colon \Sc'(\gg^*)\to\Sc'(\gg)$.


We recall that according to a theorem of Kirillov (see \cite[Chap~2]{CG90}), there exists a natural one-to-one correspondence between the coadjoint orbits and the equivalence classes of unitary irreducible representations of a nilpotent, connected and simply connected Lie group $G$. 

\begin{ex}\label{heis}
\normalfont
For every integer $n\ge1$ let $(\cdot\mid\cdot)$ denote 
the Euclidean scalar product on ${\mathbb R}^n$. 
The \emph{Heisenberg algebra} is 
${\mathfrak h}_{2n+1}={\mathbb R}^n\times{\mathbb R}^n\times{\mathbb R}$ 
with the bracket 
$[(q,p,t),(q',p',t')]=[(0,0,(p\mid q')-(p'\mid q))]$. 
The \emph{Heisenberg group} is 
${\mathbb H}_{2n+1}=({\mathfrak h}_{2n+1},\cdot)$  
with the multiplication
$x\cdot y=x+y+\frac{1}{2}[x,y]$. 
If we identify the dual space $\hg_{2n+1}^*$ with ${\mathbb R}^n\times{\mathbb R}^n\times{\mathbb R}$ 
in the usual way, then every coadjoint orbit belongs to one of the following families: 
\begin{itemize}
\item[i)] the affine hyperplanes $\Oc_\hbar={\mathbb R}^n\times{\mathbb R}^n\times\{1/\hbar\}$ with $\hbar\in\RR\setminus\{0\}$; 
\item[ii)] the singletons $\Oc_{a, b}=\{(a,b,0)\}$ with $a,b\in\RR^n$. 
\end{itemize}
For every $\hbar\in\RR\setminus\{0\}$ there is a unitary irreducible representation on the Hilbert space $L^2(\RR^n)$, 
namely 
$\pi_\hbar\colon{\mathbb H}_{2n+1}\to{\mathcal B}(L^2({\mathbb R}^n))$ 
defined by  
\begin{equation}\label{sch_eq}
(\pi_\hbar(q,p,t)f)(x)={\rm e}^{{\rm i}((p\mid x)+\frac{1}{2}(p\mid q)+\frac{t}{\hbar})}f(q+x) 
\text{ for a.e. }x\in{\mathbb R}^n
\end{equation}
for arbitrary $f\in L^2({\mathbb R}^n)$ and $(q,p,t)\in{\mathbb H}_{2n+1}$. 
This is the \emph{Schr\"o\-dinger representation} of the Heisenberg group ${\mathbb H}_{2n+1}$, and corresponds to 
the coadjoint orbit $\Oc_\hbar$ in the first family above. 
Moreover, for every $a,b\in\RR^n$ there is a unitary irreducible representation on a 1-dimensional Hilbert space, 
namely 
$\pi_{(a,b)}\colon{\mathbb H}_{2n+1}\to U(1):=\{z\in\CC\mid\vert z\vert=1\}$ 
defined by 
$$\pi_{(a,b)} (q,p,t)={\rm e}^{{\rm i}((p\mid a)+(q\mid b))}$$
for all $(q,p,t)\in{\mathbb H}_{2n+1}$, and corresponds to the orbit $\Oc_{a,b}$ in the second family.
\end{ex}

\subsection*{Weyl calculi on nilpotent Lie groups.} 
\begin{defn}[\cite{An69}, \cite{An72}]
\normalfont
Let $G=(\gg,\cdot)$ be a nilpotent Lie group. 
If $\pi\colon\gg\to\Bc(\Hc)$ is a unitary representation of $G$,  then we can use Bochner integrals 
for extending it to a linear mapping 
$$\pi\colon L^1(\gg)\to\Bc(\Hc),\quad \pi(b)v=\int\limits_{\gg}b(x)\pi(x)v\de x
\text{ if $b\in L^1(\gg)$ and $v\in\Hc$.}$$ 
\end{defn}
Let $\Tc$ be  a locally convex space  which is continuously and densely embedded in $\Hc$,  
and has the property that for every $h, \chi \in\Tc$  
we have $(\pi(\cdot)h\mid\chi)\in\Sc(\gg)$.
Then we can further extend $\pi$ to a mapping 
$$\begin{aligned}
\pi\colon\Sc'(\gg)\to\Lc(\Tc,\Tc^*),\quad 
& (\pi(u)h\mid \chi) =\langle u, (\pi(\cdot)h\mid \chi)\rangle  \\
&\text{ if $u\in\Sc'(\gg)$, $h\in\Tc$, and $\chi\in\Tc$.}
\end{aligned}$$ 
where $\Tc^*$ is the strong antidual of $\Tc$, and $(\cdot\mid\cdot)$ denotes the anti-duality between $\Tc$ and $\Tc^*$ that extends the 
scalar product of $\Hc$.  
Here we have used the notation $\Lc(\Tc,\Tc^*)$ for the space of continuous linear operators between the above spaces 
(these operators are thought of as possibly unbounded linear operators in $\Hc$).

%
In the same setting, the \emph{global Weyl calculus for the representation}~$\pi$ is the mapping
$$\OP\colon\Fc^{-1}L^1(\gg)\to\Bc(\Hc),\quad 
\OP(a)=\pi(\Fc a)$$
which further extends to 
$$\OP\colon\Sc'(\gg^*)\to\Lc(\Tc,\Tc^*),\quad \OP:=\pi\circ\Fc^{-1}.$$

\begin{ex}\label{reg}
\normalfont
Consider the (left) regular representation
$$\lambda\colon\gg\to\Bc(L^2(\gg)),\quad (\lambda(x)\phi)(y)=\phi((-x)\cdot y).$$
If we extend it as above to $\lambda\colon L^1(\gg)\to\Bc(L^2(\gg))$ then we obtain 
for $b\in L^1(\gg)$ and $\psi\in L^2(\gg)$
$$(\lambda(b)\psi)(y)=\int \limits_{\gg}b(x)\psi((-x)\cdot y)\de x=(b\ast\psi)(y)$$
hence we recover the convolution product, which makes $L^1(\gg)$ into a Banach algebra. 

The above construction of the global Weyl calculus for the regular representation 
is usually considered with  $\Tc=\Sc(\gg)$,  yielding 
$$\OP\colon\Sc'(\gg^*)\to\Lc(\Sc(\gg),\Sc'(\gg))$$ 
and the related mapping 
$$\lambda\colon\Sc'(\gg)\to\Lc(\Sc(\gg),\Sc'(\gg))$$
whose values are the (possibly unbounded) convolution operators 
on the nilpotent Lie group~$G$. 
\end{ex}

In the case of the global Weyl calculus the symbol of an operator in the representation space of an irreducible representation 
may not be uniquely determined 
on the corresponding coadjoint orbit, unlike in the case of pseudo-differential Weyl calculus on $\RR^n$ or two step nilpotent groups.
(See  Example~\ref{classical} and also \cite{Mi86}. )
The Kirillov character formula says that when $a\in \Sc(\gg^*)$, and $\pi\colon G\mapsto \Bc(\Hc)$ is a irreducible representation, then $\OP(a)$ is a trace class operator and there exists a constant that depends on the unitary class of equivalence 
of $\pi$ only, such that 
$$\Tr(\OP(a))= C\int\limits_{\Oc} a(\xi) \de \xi,$$
where $\Oc$ is the coadjoint orbit corresponding to $\pi$. 
This seems to suggest that $\OP(a)$, when $a\in C^\infty(\gg)\cap \Sc'(\gg)$, depends only on the restriction of $a$ to the coadjoint orbit $\Oc$. 
This is not always the case, as it could be seen from the next example~\cite[App. Sect. I]{Ma91}; see also \cite[Ex. N4N1, pag. 9--10]{Pe88}.

\begin{ex}\label{manch}
\normalfont
Let $\gg$ be the 4-dimensional threadlike (or filiform) Lie algebra. 
Equivalently, $\gg$ is 3-step nilpotent, 4-dimensional, and its center is 1-dimensional. 
Then $\gg$ has a Jordan-H\"older basis $\{X_1,X_2,X_3,X_4\}$ satisfying the commutation relations 
$[X_4,X_3]=[X_4,X_2]=X_1$ and $[X_j,X_k]=0$ if $1\le k<j\le 4$  with $(j,k)\not\in\{(4,3),(4,2)\}$. 
Let $\Oc$ be the coadjoint orbit of the functional 
$$\xi_0\colon\gg\to\RR,\quad \xi_0(t_1X_1+t_2X_2+t_3X_3+t_4X_4)=t_1.$$
Then $\dim\Oc=2$, and if we identify $\gg^*$ to $\RR^4$ by using the basis dual to $\{X_1,X_2,X_3,X_4\}$, 
then we have 
$\Oc=\{(1,-t,\frac{t^2}{2},s)\mid s,t\in\RR\}\subset\RR^4$. 
A unitary irreducible representation 
$\pi\colon\gg\to\Bc(L^2(\RR))$ associated with the coadjoint orbit $\Oc$ can be defined by 
$$\de\pi(X_1)=\ie\1,\quad\de\pi(X_2)=-\ie t,\quad 
\de\pi(X_3)=\frac{\ie t^2}{2},\quad\de\pi(X_3)=-\frac{\de}{\de t},$$
where we denote by $t$ both the variable of the functions in $L^2(\RR)$ 
and the operator of multiplication by this variable in  $L^2(\RR)$. 

It is clear that the function $a\colon\RR^4\to\RR$, $a(y_1,y_2,y_3,y_4)=y_4^2(2y_3-y_2^2)$, 
vanishes on the coadjoint orbit $\Oc$, and on the other hand 
it was noted in \cite[pag. 236]{Ma91} that $\OP^\pi(a)=-\frac{1}{6}\1$. 
\end{ex}

N.V.~Pedersen introduced in \cite{Pe94} an orbital Weyl calculus that is specific to a given orbit, or equivalently, to a class of unitary irreducible representations
and that, in addition, gives isomorphism between	Schwartz symbols defined on the orbit and regularizing operators defined in the space of the representation. 
The calculus may depend on the choice of a  Jordan-H\"older basis.

To describe this Weyl-Pedersen calculus we need first some notation.
Let $G=(\gg,\cdot)$ be a nilpotent Lie group of dimension $m\ge 1$ and assume that 
$\{X_1,\dots,X_m\}$ is a Jordan-H\"older basis in $\gg$;  
so for $j=1,\dots,m$ if we define $\gg_j:=\spa\{X_1,\dots,X_j\}$ then $[\gg,\gg_j]\subseteq\gg_{j-1}$, 
where $\gg_0:=\{0\}$. 
Let  $\pi\colon\gg\to\Bc(\Hc)$ be a unitary representation of $G$ 
associated with a coadjoint orbit $\Oc\subseteq\gg^*$. 
Pick $\xi_0\in\Oc$, denote 
$e:=\{j\mid X_j\not\in\gg_{j-1}+\gg_{\xi_0}\}$, and then define $\gg_e:=\spa\{X_j\mid j\in e\}$. 
We have $\gg=\gg_{\xi_0}\dotplus\gg_e$ and the mapping $\Oc\to\gg_e^*$, $\xi\mapsto\xi\vert_{\gg_e}$, 
is a diffeomorphism. 
Hence we can define an orbital Fourier transform $\Sc'(\Oc)\to\Sc'(\gg_e)$,   $a\mapsto\widehat{a}$ 
which is a linear topological isomorphism  and such that for every $a\in\Sc(\Oc)$ we have 
$$(\forall X\in\gg_e)\quad \widehat{a}(X)=\int\limits_{\Oc}\ee^{-\ie\langle\xi,X\rangle}a(\xi)\de\xi.$$
Here we have the Lebesgue measure $\de x$ on $\gg_e$ corresponding to the basis $\{X_j\mid j\in e\}$ 
and $\de\xi$ is the Borel measure on $\Oc$ such that 
the aforementioned diffeomorphism $\Oc\to\gg_e^*$ is a measure preserving mapping 
and the Fourier transform $L^2(\Oc)\to L^2(\gg_e)$ 
is unitary. 
The inverse of this orbital Fourier transform is denoted by $a\mapsto\check{a}$. 

\begin{defn}[\cite{Pe94}]\label{ped}
\normalfont With the notation above, 
the \emph{Weyl-Pedersen calculus}  associated to the unitary irreducible representation $\pi$ is the mapping
$$\Op_{\pi}\colon\Sc(\Oc)\to\Bc(\Hc),\quad 
\Op_{\pi}(a)=\int\limits_{\gg_e}\widehat{a}(x)\pi(x)\de x. $$
The space of smooth vectors $\Hc_\infty:=\{v\in\Hc\mid\pi(\cdot)v\in\Ci(\gg,\Hc)\}$ 
is dense in $\Hc$ and has the natural topology of a nuclear Fr\'echet space 
with the space of the antilinear functionals denoted by $\Hc_{-\infty}:={\Hc}_\infty^*$ 
(with the strong dual topology). 
One can show that the Weyl-Pedersen calculus extends to a linear bijective mapping 
$$\Op_{\pi} \colon\Sc'(\Oc)\to\Lc(\Hc_\infty,\Hc_{-\infty}),\quad 
(\Op_{\pi} (a) v\mid w)=\langle \widehat{a},(\pi(\cdot)v\mid w)\rangle$$
for $a\in\Sc'(\Oc)$, $v,w\in\Hc_\infty$, where in the left-hand side  $(\cdot\mid\cdot)$ denotes 
the extension of the scalar product of $\Hc$ to the sesquilinear duality pairing 
between $\Hc_\infty$ and $\Hc_{-\infty}$. 
\end{defn}

Note that in fact $\Op_{\pi}$ is associated to the coadjoint orbit corresponding to the representation
$\pi$. 
Indeed if $\pi_1$ and $\pi$ are unitary equivalent representations, $\Op_{\pi}$ and $\Op_{\pi_1}$ are defined on the same space of symbols, there is a unitary operator $U$ such that $\Op_{\pi_1}(a) = U^*\Op_{\pi}(a) U$ for every $a\in \Sc(\Oc)$, and this equality extends naturally to $a\in \Sc'(\Oc)$. 
Therefore, whenever the orbit $\Oc$ is fixed and no confussion may arise, we use the notation $\Op$ instead of $\Op_{\pi}$, for any irreducible representation corresponding to $\Oc$. 

\begin{rem}\label{link}
The map $\Sc(\gg^*) \to \Lc(\Hc_{-\infty},\Hc_{\infty})$,   $a\mapsto\OP(a)$ is surjective \cite{How84}, while 
$\Sc(\Oc) \to   \Lc(\Hc_{-\infty},\Hc_{\infty})$,   $a\mapsto\Op_\pi(a)$ is bijective \cite{Pe94}. 
In fact, it follows by \cite[Thm.2.2.7]{Pe94} that $\OP(a)=\Op_\pi(b)$, where
$$ \hat b(x) = C_{\Oc, e}\mathrm{Tr} (\pi(-x) \OP(a)), \quad x\in \gg_e,$$
and $C_{\Oc, e}$ is a constant that depends on the Jordan-H\"older basis and on the coadjoint orbit $\Oc$.
\end{rem}

\begin{ex}\label{classical}
\normalfont 
The Weyl-Pedersen calculus for the Schr\"odinger representation of the Heisenberg group ${\mathbb H}_{2n+1}$ 
is just the usual Weyl calculus from the theory of partial differential equations on $\RR^n$, 
as developed for instance in \cite{Hor79} and \cite{Hor07}. 
In this case we have $\Hc=L^2(\RR^n)$ and $\Hc_\infty=\Sc(\RR^n)$. 
\end{ex}

\section{Boundedness for the orbital Weyl calculus}\label{section3}

Let $G=(\gg,\cdot)$ be a nilpotent Lie group with a unitary irreducible representation $\pi\colon\gg\to\Bc(\Hc)$ 
associated with the coadjoint orbit $\Oc\subseteq\gg^*$. 
One proved in \cite{BB11b} that if the symbols belong to suitable modulation spaces $M^{\infty,1}\hookrightarrow\Sc'(\Oc)$, 
then the corresponding values of the Weyl-Pedersen calculus belong to $\Bc(\Hc)$. 
This condition does not require smoothness properties of symbols. 

We will now describe a result established in \cite{BB12} which extends both the classical Calder\'on-Vaillancourt theorem 
and Beals' criterion on the Weyl calculus, 
by linking growth properties of derivatives of smooth symbols 
to boundedness properties of the corresponding pseudo-differential operators in the case when the coadjoint orbit $\Oc$ is \emph{flat}, 
in the sense that for some $\xi_0\in\Oc$ we have $\Oc=\{\xi\in\gg^*\mid\xi\vert_{\zg}=\xi_0\vert_{\zg}\}$, 
where $\zg$ is the center of~$\gg$. 
This is equivalent to the condition $\dim\Oc=\dim\gg-\dim\zg$, 
and it is also equivalent to the fact that the representation $\pi$ is square integrable modulo the center of $G$. 
Since this hypothesis on coadjoint orbits might look more restrictive than it really is, 
we recall that every nilpotent Lie group embeds into a nilpotent Lie group 
whose generic coadjoint orbits (that is, the ones of maximal dimension) are flat \cite[Ex. 4.5.14]{CG90}. 

If $\Oc$ is a generic flat coadjoint orbit, 
then the Weyl-Pedersen calculus $\Op\colon\Sc'(\Oc)\to\Lc(\Hc_\infty,\Hc_{-\infty})$ 
is a linear topological isomorphism which is uniquely determined by the condition 
that for every $b\in\Sc(\gg)$ we have  
$$\Op((\Fc^{-1}b)\vert_{\Oc})=\int\limits_{\gg}\pi(x) b(x)\de x,$$ 
(see \cite[Th. 4.2.1]{Pe94}). 

We define $\Diff(\Oc)$ as the space of all linear differential operators $D$ on $\Oc$ 
which are \emph{invariant} to the coadjoint action, in the sense that   
 $$
(\forall x\in\gg)(\forall a\in C^\infty(\Oc))\quad  D(a\circ\Ad_G^*(x)\vert_{\Oc})=(Da)\circ\Ad_G^*(x)\vert_{\Oc}.$$
Let us consider the Fr\'echet space of symbols
\begin{equation}\label{symbols}
\Cc^\infty_b(\Oc) =\{a\in C^\infty(\Oc)\mid  Da\in L^\infty(\Oc)\;\text{for all} \; D\in\Diff(\Oc)\},
\end{equation} 
with the topology given by the seminorms $\{a\mapsto\Vert Da\Vert_{L^\infty(\Oc)}\}_{D\in\Diff(\Oc)}$.

\begin{thm}[\cite{BB12}]\label{main_introd}
Let $G$ be a connected, simply connected, nilpotent Lie group whose generic coadjoint orbits are flat. 
Let $\Oc$ be such an orbit with a corresponding unitary irreducible representation $\pi\colon G\to\Bc(\Hc)$.
Then for $a\in C^\infty(\Oc)$ we have 
$$a\in \Cc^\infty_b(\Oc)\iff(\forall D\in\Diff(\Oc)) \quad \Op(Da)\in\Bc(\Hc).$$ 
Moreover the Weyl-Pedersen calculus defines a continuous linear map 
$$\Op\colon \Cc^{\infty}_b(\Oc)\to\Bc(\Hc),$$
and the Fr\'echet topology of $C_b^\infty(\Oc)$ is equivalent to that defined by the family of seminorms 
$\{a\mapsto\Vert\Op(Da)\Vert\}_{D\in\Diff(\Oc)}$. 
\end{thm}

Let $\Cc^\infty_\infty(\Oc)$ be the space of all $a\in \Cc^\infty(\Oc)$ such that 
the function $Da$ vanishes at infinity on $\Oc$, for every $D\in \Diff(\Oc)$. 
It easily follows by
 the above theorem that if $a \in  \Cc^\infty(\Oc)$ then 
 $$ a\in\Cc_{\infty}^{\infty}(\Oc) \iff (\forall D\in\Diff(\Oc))\quad  \Op(Da) \text{ is a compact operator}.$$
If $\pi$ is the Schr\"odinger representation of the $(2n+1)$-dimensional 
Heisenberg group, then Theorem~\ref{main_introd} 
characterizes the symbols of type $S^{0}_{0,0}$ 
for the pseudo-differential Weyl calculus 
$\Op\colon\Sc'(\RR^{2n})\to\Lc(\Sc(\RR^n),\Sc'(\RR^n))$.  
Name\-ly,  for any symbol $a\in C^\infty(\RR^{2n})$ we have 
$$(\forall\alpha\in{\mathbb N}^{2n})\quad\partial^\alpha a\in L^\infty(\RR^{2n})\iff 
(\forall\alpha\in{\mathbb N}^{2n})\quad \Op(\partial^\alpha a)\in\Bc(L^2(\RR^n)),$$  
where  $\partial^\alpha $ stand as usually for the partial derivatives. 
The above statement unifies the Calder\'on-Vaillancourt theorem and the so-called Beals' criterion for the Weyl calculus. 

\subsection*{Application to $3$-step nilpotent Lie groups}

\begin{prop}\label{new}
Let $G=(\gg,\cdot)$ be a nilpotent Lie group with an irreducible representation $\pi\colon\gg\to\Bc(\Hc)$ 
associated with the coadjoint orbit $\Oc\subseteq\gg^*$. 
If $H=(\hg,\cdot)$ is a normal subgroup of $G$, then the following assertions are equivalent: 
\begin{enumerate}
\item\label{new_item1} The restricted representation $\pi\vert_{\hg}\colon\hg\to\Bc(\Hc)$ is irreducible. 
\item\label{new_item2} The mapping $\gg^*\to\hg^*$, $\xi\mapsto\xi\vert_{\hg}$ 
gives a diffeomorphism of $\Oc$ onto a coadjoint orbit of $H$, which will be denoted by $\Oc\vert_{\hg}$. 
\end{enumerate}
If this is the case, then the irreducible representation $\pi\vert_{\hg}$ is associated with 
the coadjoint orbit $\Oc\vert_{\hg}$ of $H$. 
\end{prop}

\begin{proof}
Pick any Jordan-H\"older sequence that contains~$\hg$. 
Then the implication $\ref{new_item1}.\Rightarrow\ref{new_item2}.$ 
follows at once by iterating \cite[Th. 2.5.1]{CG90}, 
while the converse implication follows by using Vergne polarizations. 
The details of the proof will be given elsewhere.  
\end{proof}

\begin{rem}\label{new_rem}
\normalfont
In the setting of Proposition~\ref{new}, if $X_1,\dots,X_m$ is a Jordan-H\"older basis in $\gg$ 
such that $X_1,\dots,X_k$ is a Jordan-H\"older basis in $\hg$ for  
$k=\dim\hg$, then the following assertions hold true: 
\begin{itemize}
\item The coadjoint $G$-orbit $\Oc$ and the coadjoint $H$-orbit $\Oc\vert_{\hg}$ 
have the same set of jump indices $e\subseteq\{1,\dots,k\}$. 
In particular $\gg_e=\hg_e\subseteq\hg$ and $\Oc$ cannot be flat. 
\item We have the $H$-equivariant diffeomorphism $\Theta\colon\Oc\to\Oc\vert_{\hg}$, $\xi\mapsto\xi\vert_{\hg}$.  
It intertwines the orbital Fourier transforms $\Sc'(\Oc)\to\Sc'(\gg_e)=\Sc'(\hg_e)$ 
and $\Sc'(\Oc\vert_{\hg})\to\Sc'(\hg_e)$. 
Therefore, by using also the previous remark, we obtain
\begin{equation}\label{op}
(\forall a\in\Sc'(\Oc))\quad \Op_{\pi}(a)=\Op_{\pi\vert_H}(a\circ \Theta^{-1}).
\end{equation} 
\item The $H$-equivariant diffeomorphism $\Theta$ also induces a unital injective homomorphism 
of associative algebras 
$$\Diff(\Oc)\hookrightarrow\Diff(\Oc\vert_{\hg}),\quad D\mapsto D^{\hg},$$
such that $D(a\circ\Theta)=(D^{\hg}a)\circ\Theta$ for all $a\in\Ci(\Oc\vert_{\hg})$ 
and $D\in\Diff(\Oc)$. 
In particular, it follows that for $p\in\{0\}\cup[1,\infty]$ we get a continuous injective linear map
\begin{equation}\label{symb}
\Cc^{\infty,p}(\Oc)\hookrightarrow\Cc^{\infty,p}(\Oc\vert_{\hg}),\quad a\mapsto a\circ\Theta^{-1}
\end{equation}
where we use the Fr\'echet spaces $\Cc^{\infty,p}$ introduced in \cite[Th. 4.4]{BB12}. 
\end{itemize}
In particular, 
Proposition~\ref{new},  \eqref{op} and \eqref{symb}, 
along with Theorem~\ref{main_introd} prove the next corollary. 
\end{rem}

\begin{cor}\label{newcor}
If one of the equivalent conditions  of Proposition~\ref{new} holds true and if the orbit $\Oc\vert_{\hg}$ is flat, then 
for every $a\in \Cc^\infty_b(\Oc)$ we have $\Op_{\pi}(a)\in\Bc(\Hc)$, 
and moreover the Weyl-Pedersen calculus defines a continuous linear map 
$\Op_{\pi}\colon \Cc^{\infty}_b(\Oc)\to\Bc(\Hc)$.
\end{cor}

Note that the coadjoint orbit of the representation $\pi$ in the above Corollary~\ref{newcor} 
may not be flat, and yet we have proved an $L^2$-boundedness assertion just like the one of Theorem~\ref{main_introd}. 
We now provide a specific example of a 3-step nilpotent Lie group taken from \cite{Ra85}, which illustrates this result. 

\begin{ex}\label{ratcl}
\normalfont
Recall from Example~\ref{heis} the Heisenberg algebra $\hg_{2n+1}=\RR^{2n}\times\RR$ 
with the Lie bracket given by $[(x,t),(x',t')]=(0,\omega(x,x'))$ for $x,x'\in\RR^{2n}$ and $t,t'\in\RR$, 
where $\omega\colon\RR^{2n}\times\RR^{2n}\to\RR$ is the symplectic form 
given by $\omega((q,p),(q',p'))=(p\mid q')-(p'\mid q)$ for $(q,p),(q',p')\in\RR^n\times\RR^n=\RR^{2n}$. 
It easily follows that if we consider the symplectic group 
$$\Sp(\RR^{2n},\omega)=\{T\in M_{2n}(\RR)\mid (\forall x,x'\in\RR^{2n})\quad \omega(Tx,Tx')=\omega(x,x')\}$$
then every $T\in\Sp(\RR^{2n},\omega)$ gives rise to an automorphism $\alpha_T\in\Aut(\hg_{2n+1})$ 
by the formula $\alpha_T(x,t)=(Tx,t)$ for all $x\in\RR^{2n}$ and $t\in\RR$. 
Moreover, the correspondence 
$\alpha\colon\Sp(\RR^{2n},\omega)\to\Aut(\hg_{2n+1})\simeq\Aut({\mathbb H}_{2n+1})$, $T\mapsto\alpha_T$, 
is a group homomorphism and is injective. 
On the other hand, if we write the elements of $\Sp(\RR^{2n},\omega)$ as $2\times 2$-block matrices 
with respect to the decomposition $\RR^{2n}=\RR^n\times\RR^n$, then it is well known that 
$\Sp(\RR^{2n},\omega)$ is a Lie group whose Lie algebra is 
$$\sp(\RR^{2n},\omega)=\Bigl\{
\begin{pmatrix}
\hfil A  & \hfil B \\
\hfil C & \hfil -A^\top
\end{pmatrix}\mid A, B=B^\top, C=C^\top\in M_n(\RR)\Bigr\}$$
where $B^\top$ stands for the transpose of a matrix~$B$. 
In particular, $\sp(\RR^{2n},\omega)$ has the following abelian Lie subalgebra 
$${\mathfrak s}_n(\RR)=\Bigl\{
\begin{pmatrix}
0  & 0 \\
C & 0
\end{pmatrix}\mid C=C^\top\in M_n(\RR)\Bigr\}$$
and the corresponding Lie subgroup of $\Sp(\RR^{2n},\omega)$ is the abelian Lie group
$$S_n(\RR)=\Bigl\{
\begin{pmatrix}
\1  & 0 \\
  C & \1
\end{pmatrix}\mid C=C^\top\in M_n(\RR)\Bigr\}.$$
It is easily seen that the group $S_n(\RR)$ acts (via $\alpha$) on $\hg_{2n+1}$ by unipotent automorphisms, 
and therefore the corresponding semidirect product $G=S_n(\RR)\ltimes_\alpha{\mathbb H}_{2n+1}$ is 
a nilpotent Lie group. 
Its Lie algebra is $\gg={\mathfrak s}_n(\RR)\ltimes\hg_{2n+1}$. 
If we denote by $\sym_n(\RR)$ the set of all symmetric matrices in $M_n(\RR)$ viewed as an abelian Lie algebra, 
then we have an isomorphism of Lie algebras $\gg\simeq\sym_n(\RR)\ltimes\hg_{2n+1}$ 
with the Lie bracket given by 
$$[(C,q,p,t),(C',q',p',t')]=(0,0,Cq'-C'q,(p\mid q')-(p'\mid q)) $$
for $C,C'\in\sym_n(\RR)$, $q,q',p,p'\in\RR^n$ and $t,t'\in\RR$. 
It easily follows by the above formula that $[\gg,[\gg,[\gg,\gg]]]=\{0\}$ 
hence $\gg$ is a 3-step nilpotent Lie algebra. 
Moreover, $\hg_{2n+1}\simeq\{0\}\times\hg_{2n+1}$ is an ideal of $\gg$ and 
it was proved in \cite{Ra85} that for every coadjoint $G$-orbit $\Oc\subseteq\gg^*$ of maximal dimension 
we have $\dim\Oc=n$ and moreover the mapping $\xi\mapsto\xi\vert_{\hg_{2n+1}}$ 
gives a diffeomorphism of $\Oc$ onto a coadjoint ${\mathbb H}_{2n+1}$-orbit $\Oc\vert_{\hg_{2n+1}}$. 
Thus Proposition~\ref{new} and Remark~\ref{new_rem} apply 
for the Weyl-Pedersen calculus of a unitary irreducible representation~$\pi\colon\gg\to\Bc(\Hc)$ 
associated with the coadjoint orbit~$\Oc$. 
Note that this result is similar to, and yet quite different from \cite[Th. 8.6--8.7]{Ra85}, 
inasmuch as we work here with an irreducible representation (see Definition~\ref{ped}) 
instead of the regular representation 
of~$G$ (see Example~\ref{reg}).  
\end{ex}

\end{document}